\documentclass[a4paper,10pt]{article}

\usepackage{latexsym,amssymb,amsmath,amstext,bbm,amsfonts, mathtools,
  color,mathabx,amsthm}
\usepackage{paralist} 
\usepackage[utf8]{inputenc}

\newtheorem{theorem}{Theorem}

\newtheorem{lemma}{Lemma}[section]
\newtheorem{proposition}[lemma]{Proposition}

\theoremstyle{definition}
\newtheorem{definition}[lemma]{Definition}
\newtheorem{remark}[lemma]{Remark}

\DeclareMathAlphabet{\mathpzc}{OT1}{pzc}{m}{it}
\newcommand{\smallu}{\mathpzc{u}}
\newcommand{\smallx}{\mathpzc{x}}
\newcommand{\smally}{\mathpzc{y}}

\newcommand{\R}{\mathbbm R}
\newcommand{\N}{\mathbbm N}

\newcommand{\tn}{\textnormal}

\makeatletter
\renewcommand{\@fnsymbol}[1]{\@arabic{#1 }}
\makeatother

\newcommand\unnumberedfootnote[1]{ %
        \let\temp=\thefootnote %
        \renewcommand{\thefootnote}{}%
        \footnote{#1}%
        \let\thefootnote=\temp%
        \addtocounter{footnote}{-1}}

\begin{document}

\title{\Large Marked metric measure spaces} \author{Andrej
  Depperschmidt\thanks{Abteilung f\"ur Mathematische Stochastik,
    Albert-Ludwigs University of Freiburg, Eckerstra\ss e 1, D - 79104
    Freiburg, Germany, e-mail:
    depperschmidt@stochastik.uni-freiburg.de,
    p.p@stochastik.uni-freiburg.de} $\mbox{}^,$, Andreas
  Greven\thanks{Department Mathematik, Universit\"at
    Erlangen-N\"urnberg, Bismarckstra\ss e 1 1/2, D-91054 Erlangen,
    Germany, e-mail: greven@mi.uni-erlangen.de}\; and Peter
  Pfaffelhuber$\mbox{}^{1}$\\ \date{}
}

\maketitle

\unnumberedfootnote{\emph{AMS 2000 subject classification.} 60B10,
  05C80 (Primary) 60B05, 60B12 (Secondary).}

\unnumberedfootnote{\emph{Keywords and phrases.} Metric measure space,
  Gromov metric triples, Gromov-weak topology, Prohorov metric,
  Population model}

\begin{abstract}\noindent
  A marked metric measure space (mmm-space) is a triple $(X,r,\mu)$,
  where $(X,r)$ is a complete and separable metric space and $\mu$ is
  a probability measure on $X\times I$ for some Polish space $I$ of
  possible marks. We study the space of all (equivalence classes of)
  marked metric measure spaces for some fixed $I$. It arises as state
  space in the construction of Markov processes which take values in
  random graphs, e.g.\ tree-valued dynamics describing randomly
  evolving genealogical structures in population models.

  We derive here the topological properties of the space of mmm-spaces
  needed to study convergence in distribution of random
  mmm-spaces. Extending the notion of the Gromov-weak topology
  introduced in (Greven, Pfaffelhuber and Winter, 2009), we define the
  marked Gromov-weak topology, which turns the set of mmm-spaces into
  a Polish space. We give a characterization of tightness for families of  
  distributions of random mmm-spaces and identify a convergence
  determining algebra of functions, called polynomials. 
\end{abstract}

\section{Introduction}
\label{sec:intro}
Metric spaces form a basic structure in mathematics. In probability
theory, they build a natural set-up for the possible outcomes of
random experiments. In particular, the Borel $\sigma$-algebra
generated by the topology induced by a metric space is
fundamental. Here, spaces such as $\mathbbm R^d$ (equipped with the
euclidean metric), the space of c\`adl\`ag paths (equipped with the
Skorohod metric) and the space of probability measures (equipped with
the Prohorov metric) are frequently considered. Recently, random 
metric spaces which differ from these examples, have attracted
attention in probability theory. Most prominent examples are the
description of {\em random genealogical structures} via Aldous'
\emph{Continuum Random Tree} (see \cite{Ald1993} and \cite{LeGall2005}
for many related results) or the \emph{Kingman coalescent}
\cite{Evans2000}, the \emph{Brownian map} \cite{LeGall2007} and the
connected components of the \emph{Erd\H{o}s-Renyi random graph}
\cite{ABBG09}, which are all random compact metric spaces. The former
two examples give rise to trees, which are special metric spaces,
so-called $\mathbbm R$-trees \cite{Dre84}. The latter two examples are
based on random graphs and the underlying metric coincides with the
graph metric.

In order to discuss convergence in distribution of random metric
spaces, the space of metric spaces must be equipped with a topology
such that it becomes a Polish space, i.e.\ a separable topological
space, metrizable by a complete metric. Moreover, it is necessary to
identify criteria for relative compactness in this topology, allowing
to formulate tightness criteria for families of distributions on this
space. Such topological properties of the space of compact metric
spaces have been developed  using the \emph{Gromov-Hausdorff
  topology} (see \cite{Gromov2000, BurBurIva01, EvaPitWin2006}).

Many applications deal with a {\em random evolution of metric
  spaces}. In such processes, it is frequently necessary to pick a
random point from the metric space according to some appropriate
distribution, called the sampling measure. Therefore, a (probability)
measure on the metric space must be specified and the resulting
structure including this sampling measure gives rise to 
\emph{metric measure spaces} (mm-spaces). First stochastic 
processes taking values in mm-spaces, subtree-prune and re-graft
\cite{EvansWinter2006} and the tree-valued Fleming-Viot dynamics
\cite{GrevenPfaffelhuberWinter2010} have been constructed. In
\cite{GPWmetric09} it was shown that the \emph{Gromov-weak topology}
turns the space of mm-spaces into a Polish space; see also
\cite[Chapter 3$\tfrac 12$]{Gromov2000}. Recently, random
configurations and random \emph{dynamics on metric spaces} in the form
of random graphs have been studied as well (see \cite{Durrett2007}). Two
examples are \emph{percolation} \cite{Hofstadt2010} and \emph{epidemic models
  on random graphs} \cite{DDMT2010}. 

The present paper was inspired by the study of a process of random
configurations on evolving trees \cite{DGP2011}. Such objects arise in
mathematical population genetics in the context of {\em Moran models}
or \emph{multi-type branching processes}, where the random genealogy
of a population evolves together with the (genetic) types of
individuals.
At any time the state of such a process is a \emph{marked metric measure
  space}  (mmm-space), where the measure is defined on the product of the
metric space and some fixed mark/type space; see Section~\ref{ss:mmm}. 
Slightly more complicated structures arise in the study of spatial
versions of such population models, where the mark specifies both the
genetic type and the location of an individual
\cite{GrevenSunWinter2011}.

Here we establish topological properties of the space of mmm-spaces needed 
to study {\em convergence in distribution} of {\em random mmm-spaces}. This requires an
extension of the Gromov-weak topology to the marked case (Theorem~\ref{T1}),
which is shown to be Polish (Theorem~\ref{T2}), a characterization of
tightness of distributions in that space (Theorem~\ref{T4}) and a convergence
determining set of functions in the space of probability measures on
mmm-spaces (Theorem~\ref{T5}). 

\section{Main results}
\label{sec:defin-main-results}
First,  we have to introduce some notation. For product spaces $X\times Y\times
\cdots$, we denote the projection operators by $\pi_X, \pi_Y,\dots$. For a
Polish space $E$, we denote by $\mathcal M_1(E)$ the space of probability
measures on the Borel $\sigma$-Algebra on $E$, equipped with the topology of
weak convergence, which is denoted by $\Rightarrow$. Moreover, for $\varphi:
E\to E'$ (for some other Polish space $E'$), the image measure of $\mu$ under
$\varphi$ is denoted $\varphi_\ast\mu$.

Let $\mathcal C_b(E)$ denote the set of bounded  
continuous functions on $E$ and recall that a set of functions $\Pi \subseteq 
\mathcal C_b(E)$ is \emph{separating in $\mathcal M_1(E)$} iff for all 
$E$-valued random variables $X,Y$ we have $X \stackrel{d}{=}Y$ if $\mathbf
E[\Phi(\mathcal X)] = \mathbf E[\Phi(\mathcal Y)]$ for all $\Phi\in\Pi$.
Moreover, $\Pi$ is \emph{convergence determining in $\mathcal M_1(E)$} if for
any sequence $X, X_1, X_2,\dots$ of $E$-valued random variables we have $X_n 
\xRightarrow{n\to\infty}X$ iff $\mathbf E[\Phi(\mathcal X_n)]
\xrightarrow{n\to\infty} \mathbf E[\Phi(\mathcal X)]$ for all $\Phi\in\Pi$.

\medskip 

\sloppy Here and in the whole paper the key ingredients are complete
separable metric spaces $(X,r_X), (Y,r_X),\dots$ and probability
measures $\mu_X, \mu_Y,\dots$ on $X\times I, Y\times I,\dots$ for a
{\em fixed}
\begin{equation}
\label{ag1}
\mbox{complete and separable metric space } (I, r_I),
\end{equation}
which we refer to as the \emph{mark space.}

\subsection{Marked metric measure spaces} {\bf Motivation:}
The \label{ss:mmm} present paper is motivated by genealogical 
structures in population models. Consider a population $X$ of
individuals, all living at the same time. Assume that any pair of
individuals $x,y \in X$ has a common ancestor, and define a metric on
$X$ by setting $r_X(x,y)$ as the time to the most recent common
ancestor of $x$ and $y$, also referred to as their {\em genealogical
  distance}. 
In addition, individual $x\in X$ carries some {\em mark}
$\kappa_X(x)\in I$ for some measurable function $\kappa_X$. In order
to be able to sample individuals from the population, introduce a {\em
  sampling measure} $\nu_X\in\mathcal M_1(X)$ and define
\begin{equation}\label{ag6}
  \mu_X (dx, du) := \nu_X(dx) \otimes \delta_{\kappa_X(x)} (du).
\end{equation}
Recall that most population models, such as branching processes, are
exchangeable. On the level of genealogical trees, this leads to the
following notion of equivalence of marked metric measure spaces: We
call two triples $(X,r_X,\mu_X)$ and $(Y,r_Y,\mu_Y)$ equivalent if
there is an isometry $\varphi: \text{supp}(\nu_X)\to
\text{supp}(\nu_Y)$ such that $\nu_Y = \varphi_\ast \nu_X$ and
$\kappa_Y(\varphi(x)) = \kappa_X(x)$ for all $x\in
\text{supp}(\nu_X)$, i.e.\ marks are preserved under $\varphi$.

It turns out that it requires strong restrictions on $\kappa$ to turn
the set of triples $(X,r_X,\mu_X)$ with $\mu_X$ as in \eqref{ag6} into
a Polish space (see \cite{Piotr11}). Since these restrictions are
frequently not met in applications, we pass to the larger space of
triples $(X,r_X,\mu_X)$ with general $\mu_X\in\mathcal M_1(X\times I)$
right away. This leads to the following key concept. 

\begin{definition}[mmm-spaces]%
  \label{def:mmm} \leavevmode
  \begin{asparaenum}
  \item An \emph{$I$-marked metric measure space}, or
    \emph{mmm-space}, for short, is a triple $(X, r, \mu)$ such that
    $(X,r)$ is a complete and separable metric space and
    $\mu\in\mathcal M_1(X\times I)$, where $X\times I$ is equipped
    with the product topology. To avoid set theoretic pathologies we
    assume that $X \in \mathcal B(\mathbbm R)$. In all applications we
    have in mind this is always the case.
  \item \sloppy Two mmm-spaces $(X, r_X, \mu_X), (Y, r_Y, \mu_Y)$ are
    equivalent if they are \emph{measure- and mark-preserving
      isometric} meaning that there is a measurable $\varphi: 
    \text{supp}((\pi_X)_\ast\mu_X)\to \text{supp}((\pi_Y)_\ast\mu_Y)$
    such that
    \begin{equation}\label{ag2}
      r_{X}(x,x') = r_{Y}(\varphi(x), \varphi(x')) \mbox{ for all }
      x,x'\in \text{supp}((\pi_X)_\ast\mu_X)
    \end{equation}
    and
    \begin{equation}\label{ag3}
      \widetilde\varphi_\ast \mu_X = \mu_Y \mbox{ for } \widetilde\varphi(x,u)
      = (\varphi(x),u).
    \end{equation}
    We denote the equivalence class of $(X,r,\mu)$ by
    $\overline{(X,r,\mu)}.$
  \item We introduce
    \begin{align}
      \label{eq:PP001}
      \mathbbm M^I & :=\left\{\overline{(X,r,\mu)}: (X,r,\mu) \text{
        mmm-space}\right\} 
      \end{align}
      and denote the generic elements of $\mathbbm M^I$ by $\smallx,
      \smally, \dots$.
  \end{asparaenum}
\end{definition}


\begin{remark}[Connection to mm-spaces] 
  In \label{rem:mmsp} \cite{GPWmetric09}, the space of metric measure
  spaces (mm-spaces) was considered. These are triples $(X,r,\mu)$
  where $\mu\in\mathcal M_1(X)$. Two mm-spaces $(X,r_X,\mu_X)$ and
  $(Y,r_Y,\mu_Y)$ are equivalent if $\varphi$ exists such that
  \eqref{ag2} holds. The set of equivalence classes of such mm-spaces
  is denoted by $\mathbbm M$, which is closely connected to the
  structure we have introduced in Definition~\ref{def:mmm}. Namely for
  $\smallx = \overline{(X,r,\mu)}\in\mathbbm M^I$, we set
  \begin{equation}
    \label{eq:901}
    \begin{aligned}
      \pi_1(\smallx) & := \overline{(X,r,(\pi_X)_\ast\mu)} \in
      \mathbbm M, \qquad \pi_2(\smallx) & := (\pi_I)_\ast \mu \in
      \mathcal M_1(I).
    \end{aligned}
  \end{equation}
  Note that $\pi_2(\smallx)$ is the distribution of marks in $I$ and
  $\mathbbm M$ can be identified with $\mathbbm M^I$ if $\#I=1$.
\end{remark}

~

\noindent {\bf Outline:} In Section~\ref{sec:GWT}, we state that the
marked distance matrix distribution, arising by subsequently sampling
points from an mmm-space, uniquely characterizes the mmm-space
(Theorem~\ref{T1}). Hence, we can define the marked Gromov-weak
topology based on weak convergence of marked distance matrix
distributions, which turns $\mathbbm M^I$ into a Polish space
(Theorem~\ref{T2}). Moreover, we characterize relatively compact sets
in the Gromov-weak topology (Theorem~\ref{T3}). In
Subsection~\ref{sec:randommmm} we treat our main subject,
\emph{random} mmm-spaces. We characterize tightness (Theorem
~\ref{T4}) and show that polynomials, specifying an algebra of
real-valued functions on $\mathbbm M^I$, are convergence determining
(Theorem~\ref{T5}).

The proofs of Theorems~\ref{T1} --~\ref{T5}, are given in
Sections~\ref{sec:proofs1},~\ref{ss.prooft2-3},~\ref{sec:proofs5}
and~\ref{sec:proofT5}, respectively.

\subsection{The Gromov-weak topology}
\label{sec:GWT}
Our task is to define a topology that turns $\mathbbm M^I$ into a
Polish space. For this purpose, we introduce the notion of the {\em
  marked distance matrix distribution}.

\begin{definition}[Marked distance matrix distribution]%
  \label{def:mdistmat}  \leavevmode\\
  Let $(X,r,\mu)$ be an mmm-space,
  $\smallx:=\overline{(X,r,\mu)}\in\mathbbm M^I$ and
  \begin{equation}
    \label{eq:22}
    R^{(X,r)}: \begin{cases}  (X\times I)^{\mathbbm N} &\to 
      \mathbbm R_+^{\binom{\mathbbm N}{2}}\times I^{\mathbbm N} , \\
      \big((x_k,u_k)_{k\geq 1}\big) &\mapsto
      \big(\big(r(x_k, x_l)\big)_{1\leq k<l}, (u_k)_{k\geq 1}\big).
    \end{cases} 
  \end{equation}
  The \emph{marked distance matrix distribution of $\smallx =
    \overline{(X,r,\mu)}$} is defined by 
  \begin{align}\label{eq:23}
    \nu^{\smallx} := (R^{(X,r)})_\ast \mu^{\mathbbm N } \in \mathcal
    M_1(\mathbbm R^{\binom{\mathbbm N} 2} \times I^{\mathbbm N}).
  \end{align} 
  For generic elements in $\mathbbm R^{\binom{\mathbbm N}{2}}$ and
  $I^{\mathbbm N}$, we write $\underline{\underline r}$ and $\underline
  u$, respectively. 
\end{definition}

\noindent
In the above definition $(R^{(X,r)})_\ast \mu^{\mathbbm N}$ does not
depend on the particular element $(X,r,\mu)$ of the equivalence class
$\smallx=\overline{(X,r,\mu)}$, i.e.\ $\nu^\smallx$ is well-defined.
The key property of $\mathbb{M}^I$ is that the distance matrix
distribution uniquely determines mmm-spaces as the next result shows.


\begin{theorem}
  \label{T1}
  Let $\smallx, \smally\in\mathbbm M^I$. Then, $\smallx = \smally$ iff
  $\nu^\smallx = \nu^\smally$.
\end{theorem}

This characterization of elements in $\mathbb{M}^I$ allows us to
introduce a topology as follows.

\begin{definition}[Marked Gromov-weak topology]%
  \label{def:mGw} \leavevmode\\
  Let $\smallx, \smallx_1,\smallx_2,\dots\in\mathbbm M^I$. We say that
  $\smallx_n \xrightarrow{n\to\infty} \smallx$ in the \emph{marked
    Gromov-weak topology (MGW topology)} iff
  \begin{align}\label{de1}
    \nu^{\smallx_n} \xRightarrow{n\to\infty} \nu^\smallx
  \end{align}
  in the weak topology on $\mathcal M_1\bigl(\mathbbm R_+^{\binom{\mathbbm
      N}{2}} \times I^{\mathbbm N}\bigr)$, where, as usual, $\mathbbm 
  R_+^{\binom{\mathbbm N}{2}} \times I^{\mathbbm N}$ is equipped with
  the product topology of $\mathbbm R_+$ and $I$, respectively.
  \hfill\qed
\end{definition}

\noindent
The next result implies that $\mathbbm M^I$ is a suitable space to
apply standard techniques of probability theory (most importantly,
weak convergence and martingale problems).
\begin{theorem}
  \label{T2}
  The space $\mathbbm M^I$, equipped with the MGW topology, is Polish.
\end{theorem}

\noindent
In order to study weak convergence in $\mathbbm M^I$, knowledge about
relatively compact sets is crucial. 

\begin{theorem}[Relative compactness in the MGW topology]
  \label{T3}\leavevmode\\
  For $\Gamma\subseteq \mathbbm M^I$ the following assertions are
  equivalent:
  \begin{asparaenum}
  \item[(i)] The set $\Gamma$ is relatively compact with respect to
    the marked Gromov-weak topology.
  \item[(ii)] Both, $\pi_1(\Gamma)$ is relatively compact with
    respect to the Gromov-weak topology on $\mathbbm M$ and
    $\pi_2(\Gamma)$ is relatively compact with respect to the weak
    topology on $\mathcal M_1(I)$.
  \end{asparaenum}
\end{theorem}

\begin{remark}[Relative compactness in $\mathbbm M$] 
  For the application \label{rem:relcpM} of Theorem \ref{T3}, it is
  necessary to characterize relatively compact sets in $\mathbbm M$,
  equipped with the Gromov-weak topology. Proposition 7.1 of
  \cite{GPWmetric09} gives such a characterization which we recall:
  Let $r_{12}: (\underline{\underline r}, \underline u)\mapsto
  r_{12}$.  Then the set $\pi_1(\Gamma)$ is relatively compact in
  $\mathbbm M$, iff
\begin{equation}\label{ag7}
  \{(r_{12})_\ast \nu^{\smallx}: \smallx
  \in \Gamma\}\subseteq \mathcal M_1(\mathbbm R_+) \mbox{ is tight}
  \end{equation}
  and
  \begin{align}\label{eq:100}
    \sup_{\smallx = \overline{(X,r,\mu)}\in\Gamma} \mu((x,u)\in
    X\times I: \mu(B_{\varepsilon}(x)\times I)\leq
    \delta)\xrightarrow{\delta\to 0}0
  \end{align}
  for all $\varepsilon>0$, where $B_{\varepsilon}(x)$ is the open
  $\varepsilon$-ball around $x\in X$.
\end{remark}



\subsection{Random mmm-spaces}
\label{sec:randommmm}
When showing convergence in distribution of a sequence of random mmm-spaces,
it must be established that the sequence of distributions is tight and all
potential limit points are the same and hence we need (i) tightness criteria
(see Theorem \ref{T4}) and (ii) a separating (or even 
  convergence-determining) algebra of functions in $\mathcal
M_1(\mathbb{M}^I)$ (see Theorem \ref{T5}).

\begin{theorem}[Tightness of distributions on $\mathbbm M^I$]%
  \label{T4} \leavevmode\\
  For an arbitrary index set $J$ let $\{\mathcal X_j: j\in J\}$ be a family of
  $\mathbbm M^I$-valued random variables. The set of distributions of
  $\{\mathcal X_j: j\in J\}$ is tight iff
  \begin{enumerate}
  \item[(i)] the set of distributions of $\{\pi_1(\mathcal X_j): j\in J\}$ is
    tight as a subset of $\mathcal M_1(\mathbbm M)$,
  \item[(ii)] the set of distributions of $\{\pi_2(\mathcal X_n): n\in\mathbbm
    N\}$ is tight as a subset of $\mathcal M_1(I)$.
  \end{enumerate} 
\end{theorem}


\noindent 

In order to define a separating algebra of functions in $\mathcal M_1(\mathbbm
M^I)$, we denote by
\begin{equation}\label{ag8}
  \overline{\mathcal
    C}^{(k)}_n:=\overline{\mathcal C}^{(k)}_{n}(\mathbbm
  R_+^{\binom{\mathbbm N}{2}}\times I^{\mathbbm N})
\end{equation}
the set of bounded, real-valued functions $\phi$ on $\mathbbm
R_+^{\binom{\mathbbm N}{2}} \times I^{\N}$, which are $k$ times continuously
differentiable with respect to the coordinates in $\mathbbm
R_+^{\binom{\mathbbm N}{2}}$ and such that $(\underline{\underline r},
\underline u)\mapsto\phi(\underline{\underline r}, \underline u)$ depends on
the first $\binom n 2$ variables in $\underline{\underline r}$ and the first
$n$ in $\underline u$. (The space $\overline{\mathcal C}_0$ consist of
constant functions.) For $k=0$, we set $ \overline{\mathcal C}^{}_n
:=\overline{\mathcal C}^{(0)}_n$.

\begin{definition}[Polynomials]\mbox{} 
  \begin{asparaenum}
  \item A function $\Phi: \mathbbm M^I\to\mathbbm R$ is a \emph{polynomial},
    if, for some $n\in\mathbbm N_0$, there exists $\phi\in\overline{\mathcal
      C}_n$, such that
    \begin{align}\label{eq:phi}
      \Phi(\smallx) = \langle \nu^\smallx, \phi\rangle := \int
      \phi(\underline{\underline r}, \underline u)
      \nu^\smallx(d\underline{\underline r}, d\underline u)
    \end{align}
    for all $\smallx\in\mathbbm M^I$.  We then write $\Phi = \Phi^{n,\phi}$.
  \item For a polynomial $\Phi$ the smallest number $n$ such that there exists
    $\phi\in\overline{\mathcal C}_n$ satisfying \eqref{eq:phi} is called the
    \emph{degree} of $\Phi$.
  \item We set for $k=0,1,\dots,\infty$
    \begin{equation}
      \label{eq:201}
      \begin{aligned}
        \Pi^k& :=\bigcup_{n=0}^\infty \Pi^k_n, \qquad & \Pi_n^k & :=
        \{\Phi^{n,\phi}: \phi \in \overline{\mathcal C}^{(k)}_{n}\}.
      \end{aligned}
    \end{equation}
  \end{asparaenum}
\end{definition}

The following result shows that polynomials are not only separating, but even
convergence determining in $\mathcal M_1(\mathbbm M^I)$.
\begin{theorem}[Polynomials are convergence determining in
  \mbox{$\mathcal M_1(\mathbbm M^I)$}]\mbox{}
  \label{T5}
  \begin{asparaenum}
  \item For every $k=0,1,\dots,\infty$, the algebra $\Pi^k$ is separating in
    $\mathcal M_1(\mathbbm M^I)$.
  \item There exists a countable algebra $\Pi_\ast^\infty \subseteq
    \Pi^\infty$ that is convergence determining in $\mathcal M_1(\mathbbm
    M^I)$.
  \end{asparaenum}
\end{theorem}

\begin{remark}[Application to random mmm-spaces]\mbox{}
  \begin{asparaenum}
  \item In order to show convergence in distribution of random mmm-spaces
    $\mathcal X_1, \mathcal X_2,\dots$, there are two strategies. (i) If a
    limit point $\mathcal X$ is already specified, the property $\mathbf
    E[\Phi(\mathcal X_n)] \xrightarrow{n\to\infty}\mathbf E[\Phi(\mathcal X)]$
    for all $\Phi \in \Pi^k$ suffices for convergence $\mathcal X_n
    \xRightarrow{n\to\infty}\mathcal X$ by Theorem~\ref{T5}. (ii) If no limit
    point is identified yet, tightness of the sequence implies existence of
    limit points. Then, convergence of $\mathbf E[\Phi(\mathcal X_n)]$ as a
    sequence in $\mathbbm R$ for all $\Phi\in\Pi^k$ shows uniqueness of the
    limiting object. Both situations arise in practice; see the proof of
    Theorem 1(c) in \cite{DGP2011} for an application of the former and the
    proof of Theorem 4 in \cite{DGP2011} for the latter.
  \item Theorem~\ref{T5} extends Corollary 3.1 of \cite{GPWmetric09} in the
    case of unmarked metric measure spaces. As the theorem shows, convergence
    of polynomials is enough for convergence in the Gromov-weak topology if
    the limit object is known. We will show in the proof that convergence of
    polynomials 
    is enough to ensure tightness of the sequence.
  \end{asparaenum}
\end{remark}

\section{Properties of the marked Gromov-weak topology}
\label{sec:proofs14}
After proving Theorem~\ref{T1} in Section~\ref{sec:proofs1}, we introduce the
{\em Gromov-Prohorov metric} on $\mathbbm M^I$ a concept of
  interest also by itself in Section~\ref{sec:proofs23}. We will show in the
proofs of Theorems~\ref{T2} and~\ref{T3} in Section~\ref{ss.prooft2-3} that
this metric is complete and metrizes the MGW topology.

\subsection{Proof of Theorem~\ref{T1}}
\label{sec:proofs1}
We adapt the proof of Gromov's reconstruction theorem for metric measure
spaces, given by A.~Vershik -- see Chapter~3$\frac12$.5 and 3$\frac12$.7
in~\cite{Gromov2000} -- to the marked case.

Let $\smallx = \overline{(X,r_X,\mu_X)},\smally = \overline{(Y,r_Y,\mu_Y)} \in
\mathbbm M^I$. It is clear that $\nu^\smallx = \nu^\smally$ if $\smallx =
\smally$. Thus, it remains to show that the converse is also true, i.e.\ we
need to show that $\nu^{\smallx} = \nu^{\smally}$ implies that $ \smallx$ and
$\smally$ are measure-preserving isometric (see Definition~\ref{def:mmm}).

If $\nu^\smallx = \nu^\smally$, then there exists a measure $\mu\in\mathcal
M_1\big((X\times I)^\N \times (Y\times I)^\N\big)$ such that
we have (recall \eqref{eq:22}) that $(R^{(X,r)}\circ \pi_{(X\times
  I)^{\mathbbm N}})_\ast \mu = \nu^\smallx$, $(R^{(Y,r)}\circ\pi_{(Y\times
  I)^{\mathbbm N}} )_\ast \mu = \nu^\smally$, and
\begin{align}\label{eq:piXpiY}
  R^{(X,r)}\circ \pi_{(X\times I)^{\mathbbm N}}((\underline x, \underline u),
  (\underline y, \underline v))= R^{(Y,r)}\circ \pi_{(Y\times I)^{\mathbbm
      N}}((\underline x, \underline u), (\underline y, \underline v)),
\end{align}
for $\mu$-almost all $((\underline x, \underline u), (\underline y, \underline
v)) = (((x_1, u_1),(x_2,u_2),\dots), ((y_1,v_1),(y_2,v_2),\dots))$. Then in
particular, by the Glivenko-Cantelli theorem, for $\mu$-almost all
$((\underline x, \underline u), (\underline y, \underline v))$,
\begin{align}\label{eq:glivCan}
  \frac1n \sum_{k=1}^n \delta_{(x_k,u_k)} \xRightarrow{n\to\infty} \mu_X
  \quad\text{ and }\quad \frac1n \sum_{k=1}^n \delta_{(y_k,v_k)}
  \xRightarrow{n\to\infty} \mu_Y.
\end{align}

Now, take any $((\underline x, \underline u), (\underline y, \underline v))$
such that \eqref{eq:piXpiY} and \eqref{eq:glivCan} hold as well as
$(x_n,u_n)\in\text{supp}(\mu_X), (y_n,v_n)\in\text{supp}(\mu_Y)$, $n\in\N$. By
\eqref{eq:piXpiY} we find that $\underline u = \underline v$. Define $\varphi:
\text{supp}((\pi_X)_\ast\mu_X)\to \text{supp}((\pi_Y)_\ast\mu_Y)$ as the only
continuous map satisfying $\varphi(x_n) = y_n$, $n\in\mathbbm N$. By
\eqref{eq:piXpiY}, we obtain that $r_X(x_m, x_n) = r_Y(y_m,y_n) =
r_Y(\varphi(x_m), \varphi(x_n))$, $m,n\in\mathbb N$, which extends to
$\text{supp}((\pi_X)_\ast\mu_X)$ by continuity. In addition, by
\eqref{eq:piXpiY} and continuity, $\widetilde\varphi_\ast\mu_X = \mu_Y$ and so
$(X,r_X,\mu_X)$ and $(Y,r_Y,\mu_Y)$ are measure-reserving isometric, i.e.\
$\smallx = \smally$.

\subsection{The Gromov-Prohorov metric}
\label{sec:proofs23}
In this section, we define the marked Gromov-Prohorov metric on $\mathbbm
M^I$, which generates a topology which is at least as strong as the marked
Gromov-weak topology, see Lemma~\ref{appl:3}. However, since we establish in
Proposition \ref{appP:relcp1} that both topologies have the same compact sets,
we see in Proposition~\ref{prop:topsame} that the topologies are the same, and
hence, the marked Gromov-Prohorov metric metrizes the marked Gromov-weak
topology. We use the same notation for $\varphi$ and $\widetilde\varphi$ as in
Definition~\ref{def:mmm}. Recall that the topology of weak convergence of
probability measures on a separable space is metrized by the Prohorov metric
(see \cite[Theorem~3.3.1]{EthierKurtz86}).



\begin{definition}[The marked Gromov-Prohorov topology]%
  \leavevmode\\
  \label{def:MGW}
  For $\smallx_i = \overline{(X_i,r_i, \mu_i)}\in\mathbbm M^I$, $i=1,2$, set
  \begin{align}\label{appeq:1}
    d_{\tn{MGP}}(\smallx_1, \smallx_2) := \inf_{(Z, \varphi_1, \varphi_2)}
    d_{\tn{Pr}}((\widetilde\varphi_1)_\ast \mu_1, (\widetilde\varphi_2)_\ast
    \mu_2),
  \end{align}
  where the infimum is taken over all complete and separable metric spaces
  $(Z,r_Z)$, isometric embeddings $\varphi_1: X_1\to Z$, $\varphi_2:X_2\to Z$
  and $d_{\tn{Pr}}$ denotes the Prohorov metric on $\mathcal M_1(Z\times I)$,
  based on the metric $\widetilde r_Z = r_Z + r_I$ on $Z\times I$, metrizing
  the product topology. Here, $d_{\tn{MGP}}$ denotes the \emph{marked
    Gromov-Prohorov metric (MGP metric)}. The topology induced by
  $d_{\tn{MGP}}$ is called the \emph{marked Gromov-Prohorov topology (MGP
    topology)}.
\end{definition}

\begin{remark}[Equivalent definition of the MGP metric]%
  \leavevmode\\ 
  \label{rem:eqMGP}
  For $\smallx_i = \overline{(X_i,r_i, \mu_i)} \in \mathbbm M^I$, $i=1,2$,
  denote by $X_1\sqcup X_2$ the disjoint union of $X_1$ and $X_2$. Then,
  \begin{align}\label{appeq:2}
    d_{\tn{MGP}}(\smallx_1, \smallx_2) := \inf_{r_{X_1\sqcup X_2}}
    d_{\tn{Pr}}((\widetilde\varphi_1)_\ast \mu_1, (\widetilde\varphi_2)_\ast
    \mu_2),
  \end{align}
  where the infimum is over all metrics $r_{X_1\sqcup X_2}$ on $X_1\sqcup X_2$
  extending the metrics on $X_1$ and $X_2$ and $\varphi_i: X_i\to X_1\sqcup
  X_2$, $i=1,2$ denote the canonical embeddings.
\end{remark}

\begin{remark}[$d_{\mathrm{MGP}}$ is a metric]  
  The \label{rem:met2} fact that $d_{\tn{MGP}}$ indeed defines a metric
  follows from an easy extension of Lemma~5.4 in \cite{GPWmetric09}. While
  symmetry and positive definiteness are clear from the definition, the
  triangle inequality holds by the following argument: For three mmm-spaces
  $\smallx_i=\overline{(X_i,r_i,\mu_i)} \in\mathbbm M^I$, $i=1,2,3$ and any
  $\varepsilon>0$, by the same construction as in Remark~\ref{rem:eqMGP}, we
  can choose a metric $r_{X_1\sqcup X_2\sqcup X_3}$ on $X_1\sqcup X_2 \sqcup
  X_3$, extending the metrics $r_{X_1}, r_{X_2}, r_{X_3}$, such that
  \begin{equation}
    \label{appeq:12}
    \begin{aligned}
      d_{\text{Pr}}((\widetilde \varphi_1)_\ast\mu_1, (\widetilde
      \varphi_2)_\ast\mu_2) - d_{\tn{MGP}}(\smallx_1, \smallx_2) & <
      \varepsilon, \\ d_{\text{Pr}}((\widetilde \varphi_2)_\ast\mu_2,
      (\widetilde \varphi_3)_\ast\mu_3- d_{\tn{MGP}}(\smallx_2, \smallx_3) & <
      \varepsilon.
    \end{aligned}
  \end{equation}
  Then, we can use the triangle inequality for the Prohorov metric on
  $\mathcal M_1((X_1\sqcup X_2 \sqcup X_3) \times I)$ and let $\varepsilon\to
  0$ to obtain the triangle inequality for $d_{\tn{MGP}}$.
\end{remark}

\begin{lemma}[\sloppy Equivalent description of the MGP topology]%
  \leavevmode\\
  Let \label{appl:equiv}$\smallx=\overline{(X,r_X,\mu_X)}, \smallx_1 =
  \overline{(X_1,r_1,\mu_1)}, \smallx_2 = \overline{(X_2,r_2,\mu_2)}, \ldots
  \in \mathbbm M^I$. Then, $d_{\tn{MGP}}(\smallx_n, \smallx)
  \xrightarrow{n\to\infty} 0$ if and only if there is a complete and separable
  metric space $(Z,r_Z)$ and isometric embeddings $\varphi_X: X\to Z,
  \,\varphi_1: X_1\to Z,\, \varphi_2: X_2\to Z,\dots$ with
  \begin{align}\label{appeq:5}
    d_{\tn{Pr}}((\widetilde\varphi_n)_\ast \mu_n, (\widetilde\varphi_X)_\ast
    \mu_X) \xrightarrow{n\to\infty} 0.
  \end{align}
\end{lemma}
\begin{proof}
  The assertion is an extension of Lemma 5.8 in \cite{GPWmetric09} to the
  marked case. The proof of the present lemma follows the same lines, which we
  sketch briefly.

  First, the ``if''-direction is clear. For the ``only if'' direction, fix a
  sequence $\varepsilon_1, \varepsilon_2,\dots>0$ with $\varepsilon_n\to 0$ as
  $n\to\infty$. By the same construction as in Remark~\ref{rem:met2}, we can
  construct a metric $r_Z$ on $Z= X\sqcup X_1 \sqcup X_2\sqcup\cdots$ with the
  property that
  \begin{align}\label{appeq:6}
    d_{\tn{Pr}}((\widetilde\varphi_n)_\ast\mu_n,
    (\widetilde\varphi_X)_\ast\mu_X) - d_{\tn{MGP}}(\smallx_n,
    \smallx) < \varepsilon_n,
  \end{align}
  where $\varphi_X: X\to Z$ and $\varphi_n: X_n\to Z, n\in\mathbbm N$ are
  canonical embeddings. The assertion follows.
\end{proof}

\begin{lemma}[MGP \, convergence \, implies \, MGW convergence]%
  \leavevmode\\
  Let \label{appl:3} $\smallx, \smallx_1,\smallx_2,\dots\in\mathbbm M^I$ be
  such that $d_{\tn{MGP}}(\smallx_n, \smallx) \xrightarrow{n\to\infty}
  0$. Then, $\smallx_n\xrightarrow{n\to\infty}\smallx$ in the MGW topology.
\end{lemma}
\begin{proof}
  Let $\smallx= \overline{(X,r,\mu)}, \smallx_1 = \overline{(X_1, r_1,
    \mu_1)}, \smallx_2 = \overline{(X_2, r_2, \mu_2)},\dots$. Take $(Z,r_Z)$
  and isometric embeddings $\varphi_X, \varphi_1, \varphi_2,\dots$ such that
  \eqref{appeq:5} from Lemma~\ref{appl:equiv} holds.

  It is a consequence of Proposition 3.4.5 in \cite{EthierKurtz86} that
  $\bigcup_n \overline{\mathcal C}_n$ is convergence determining in $\mathcal
  M_1(\mathbbm R_+^{\binom{\mathbbm N}{2}} \times I^{\mathbbm N})$; see also
  the proof of Proposition \ref{P:firstSecond1}. Let $\Phi\in\Pi^0$ be such
  that $\Phi(.) = \langle \nu^., \phi\rangle$ for some $\phi\in
  \bigcup_{n=0}^\infty\overline{\mathcal C}_n$. Since
  $(\widetilde\varphi_n)_\ast\mu_n \xRightarrow{n\to\infty}
  (\widetilde\varphi_X)_\ast\mu_X$ by \eqref{appeq:5}, we also have that
  $\big((\widetilde\varphi_n)_\ast\mu_n\big)^{\otimes \mathbbm N}
  \xRightarrow{n\to\infty} \big((\widetilde\varphi_X)_\ast\mu_X\big)^{\otimes
    \mathbbm N}$ in $\mathcal M_1((Z\times I)^{\mathbbm N})$. Hence we can
  conclude that
  \begin{equation}
    \label{appeq:8}
    \begin{aligned}
      \int \phi\big( (r_Z(z_k, z_l))_{1\leq k< l}, & \underline u\big)
      \big((\widetilde\varphi_n)_\ast\mu_n\big)^{\otimes \mathbbm N}
      (d\underline z, d\underline u) \\ & \xrightarrow{n\to\infty} \int
      \phi\big( (r_Z(z_k, z_l))_{1\leq k< l}, \underline u\big)
      \big((\widetilde\varphi_X)_\ast\mu_X\big)^{\otimes \mathbbm
        N}(d\underline z, d\underline u).
    \end{aligned}
  \end{equation}
  Since $\smallx = \overline{(Z,r_Z,(\widetilde\varphi_X)_\ast\mu_X)}$ and
  $\smallx_n = \overline{(Z,r_Z,(\widetilde\varphi_n)_\ast\mu_n)},
  n=1,2,\dots$, this proves that $\langle\nu^{\smallx_n}, \phi\rangle
  \xrightarrow{n\to\infty}\langle\nu^{\smallx}, \phi\rangle$. Because
  $\Phi\in\Pi^0$ was arbitrary, we have that $\nu^{\smallx_n}
  \xRightarrow{n\to\infty}\nu^\smallx$. Then, by definition,
  $\smallx_n\xrightarrow{n\to\infty}\smallx$ in the MGW topology.
\end{proof}

\begin{proposition}[Relative compactness in $\mathbbm M^I$]%
  \label{appP:relcp1}\leavevmode\\
  Let $\Gamma\subseteq \mathbbm M^I$. Then conditions (i) and (ii) of
  Theorem~\ref{T3} are equivalent to
  \begin{asparaenum}
  \item[(iii)] The set $\Gamma$ is relatively compact with respect to the
    marked Gromov-Prohorov topology.
  \end{asparaenum}
\end{proposition}
\begin{proof}
  First, (iii)$\Rightarrow$(i) follows from Lemma~\ref{appl:3}. Thus, it
  remains to show (i)$\Rightarrow$(ii)$\Rightarrow$(iii).

  \medskip 
  \noindent
  (i)$\Rightarrow$(ii): Note that $\Pi^0$ contains functions $\Phi(.)  =
  \langle\nu^., \phi\rangle$ such that $\phi$ does not depend on the variables
  $\underline u\in I^{\mathbbm N}$, as well as functions $\phi$ which only
  depend on $u_1\in I$. Denote the former set of functions by
  $\Pi_{\tn{dist}}$ and the latter by $\Pi_{\tn{mark}}$.

  Assume that the sequence $\smallx_1, \smallx_2,\dots\in \Gamma$ converges to
  $\smallx\in\mathbbm M^I$ with respect to the MGW topology. Since
  $\Phi(\smallx_n) \xrightarrow{n\to\infty} \Phi(\smallx)$ for all
  $\Phi\in\Pi_{\tn{dist}}$, we find that $\pi_1(\smallx_n)
  \xrightarrow{n\to\infty} \pi_1(\smallx)$ in the Gromov-weak topology. In
  addition, $\Phi(\smallx_n) \xrightarrow{n\to\infty} \Phi(\smallx)$ for all
  $\Phi\in\Pi_{\tn{mark}}$ implies $\pi_2(\smallx_n) \xRightarrow{n\to\infty}
  \pi_2(\smallx)$. In particular, (ii) holds.

  \medskip 
  
  \noindent
  (ii)$\Rightarrow$(iii): Recall from Theorem 5 of \cite{GPWmetric09} that the
  (unmarked) Gromov-weak and the (unmarked) Gromov-Prohorov topology
  coincide. For a sequence in $\Gamma$, take a subsequence $\smallx_1 =
  \overline{(X_1, r_1, \mu_1)}, \smallx_2 = \overline{(X_2, r_2,
    \mu_2)},\dots\in\Gamma$ and $\smallx = \overline{(X,r_X,\mu_X)}\in\mathbbm
  M^I$ such that $\pi_1(\smallx_n)\xrightarrow{n\to\infty}
  \pi_1(\smallx)\in\mathbbm M$ in the Gromov-Prohorov topology and
  \begin{align}\label{appeq:9}
    d_{\tn{Pr}}(\pi_2(\smallx_n), \pi_2(\smallx)) \xrightarrow{n\to\infty} 0.
  \end{align}
  Using Lemma 5.7 of \cite{GPWmetric09}, take a complete and separable metric
  space $(Z, r_Z)$, isometric embeddings $\varphi_X: X\to Z, \varphi_1: X_1\to
  Z, \varphi_2: X_2\to Z,\dots$ such that
  \begin{equation}
    \label{appeq:10}
    \begin{aligned}
      d_{\tn{Pr}}((\pi_{X_n}\circ\widetilde\varphi_n)_\ast\mu_n,&
      (\pi_X\circ\widetilde\varphi_X)_\ast\mu_X) \\ & =
      d_{\tn{Pr}}((\pi_{X_n})_\ast((\widetilde\varphi_n)_\ast\mu_n),
      (\pi_X)_\ast((\widetilde\varphi_X)_\ast\mu_X))\xrightarrow{n\to\infty}
      0.
    \end{aligned}
  \end{equation}
  In particular, \eqref{appeq:9} shows that $\{\pi_2(\smallx_n)=
  (\pi_I)_\ast(\widetilde\varphi_n)_\ast \mu_n: n\in\mathbbm N\}$ is
  relatively compact in $\mathcal M_1(I)$ and \eqref{appeq:10} shows that $\{
  (\pi_{X_n})_\ast((\widetilde\varphi_n)_\ast \mu_n): n\in\mathbbm N\}$ is
  relatively compact in $\mathcal M_1(Z)$. This implies that $\{
  (\widetilde\varphi_n)_\ast \mu_n: n\in\mathbbm N\}$ is relatively compact in
  $\mathcal M_1(Z\times I)$. Hence, we can find a convergent subsequence, and
  (iii) follows by Lemma~\ref{appl:equiv}.
\end{proof}

\begin{proposition}[MGW and MGP topologies coincide]%
  \leavevmode\\
  The \label{prop:topsame} \sloppy marked Gromov-Prohorov metric generates the
  marked Gromov-weak topology, i.e.\ the marked Gromov-weak topology and the
  marked Gromov-Prohorov topology coincide.
\end{proposition}
\begin{proof}
  Let $\smallx, \smallx_1, \smallx_2,\dots\in\mathbbm M^I$. We have to show
  that $\smallx_n\xrightarrow{n\to\infty} \smallx$ in the MGW topology if and
  only if $\smallx_n\xrightarrow{n\to\infty} \smallx$ in the MGP topology. The
  'if'-part was shown in Lemma~\ref{appl:3}.  For the 'only if'-direction,
  assume that $\smallx_n\xrightarrow{n\to\infty} \smallx$ in the MGW
  topology. It suffices to show that for all subsequences of
  $\smallx_1,\smallx_2,\dots$, there is a further subsequence $\smallx_{n_1},
  \smallx_{n_2},\dots$ such that
  \begin{align}\label{appeq:14}
    d_{\tn{MGP}}(\smallx_{n_k}, \smallx)\xrightarrow{k\to\infty} 0.
  \end{align}
  By Proposition~\ref{appP:relcp1} $\{\smallx_n: n\in\mathbbm N\}$ is
  relatively compact in the marked Gromov-Prohorov topology. Therefore, for a
  subsequence, there exists $\smally\in\mathbbm M^I$ and a further subsequence
  $\smallx_{n_1}, \smallx_{n_2},\dots$ with
  $\smallx_{n_k}\xrightarrow{k\to\infty}\smally$ in the Gromov-Prohorov
  topology. By the 'if'-direction it follows that
  $\smallx_{n_k}\xrightarrow{k\to\infty}\smally$ in the Gromov-weak topology,
  which shows that $\smally = \smallx$ and therefore \eqref{appeq:14} holds.
\end{proof}

\noindent
\subsection{Proofs of Theorems~\ref{T2} and~\ref{T3}}
\label{ss.prooft2-3}
Clearly, Theorem~\ref{T3} was already shown in Proposition~\ref{appP:relcp1}.

For Theorem~\ref{T2}, some of our arguments are similar to proofs in
\cite{GPWmetric09}, where the case without marks is treated, which are also
based on a similar metric. We have shown in Proposition~\ref{prop:topsame}
that the marked Gromov-Prohorov metric metrizes the marked Gromov-weak
topology. Hence, we need to show that the marked Gromov-weak topology is
separable, and $d_{\text{MGP}}$ is complete.

We start with separability. Note that the Gromov-Prohorov topology coincides
with the topology of weak convergence on $\{\nu^\smallx: \smallx\in\mathbbm
M^I\}\subseteq \mathcal M_1(\mathbbm R_+^{\binom{\mathbbm N}{2}}\times
I^{\mathbbm N})$. Hence, separability follows from separability of the
topology of weak convergence on $\mathcal M_1(\mathbbm R_+^{\binom{\mathbbm
    N}{2}}\times I^{\mathbbm N})$.
  
For completeness, consider a Cauchy sequence $\smallx_1,\smallx_2, \dots \in
\mathbbm M^I$. It suffices to show that there is a convergent
subsequence. Note that $\pi_1(\smallx_n)$ is Cauchy in $\mathbbm M$ and
$\pi_2(\smallx_n)$ is Cauchy in $\mathcal M_1(I)$. In particular,
$\{\pi_i(\smallx_n): n\in\mathbbm N\}, i=1,2$ are relatively compact. By
Proposition~\ref{appP:relcp1}, this implies that $\{\smallx_n: n\in\mathbbm
N\}$ is relatively compact in $\mathbbm M^I$ and thus, there exists a
convergent subsequence.

\section{Properties of random mmm-spaces}
\label{s.randommmm}
In this section we prove the probabilistic statements which we asserted in
Subsection \ref{sec:randommmm}. In particular, we prove Theorems~\ref{T4} in
Section~\ref{sec:proofs5} and Theorem~\ref{T5} in
Section~\ref{sec:proofT5}. In Section \ref{sec:proofs5poly} we give properties
of polynomials a class of functions not only crucial for the
  topology of $\mathbb{M}^I$ but also to formulate martingale problems (see
  \cite{DGP2011,GrevenPfaffelhuberWinter2010}).

\subsection{Proof of Theorem~\ref{T4}}
\label{sec:proofs5}
The proof is an easy consequence of Theorem~\ref{T3}: By Prohorov's Theorem,
the family of distributions of $\{\mathcal X_j: j\in J\}$ is tight iff for all
$\varepsilon>0$ there is $\Gamma_\varepsilon\subseteq \mathbbm M^I$ relatively
compact with $\inf_{j\in J} \mathbf P(\mathcal
X_j\in\Gamma_\varepsilon)>1-\varepsilon$. By Theorem~\ref{T3} the latter is
the case iff for all $\varepsilon>0$ there are relatively compact
$\Gamma_\varepsilon^1\subseteq \mathbbm M$ and $\Gamma_\varepsilon^2\subseteq
\mathcal M_1(I)$ such that
\begin{align}\label{appeq:11}
  \inf_{j\in J} \mathbf P(\pi_1(\mathcal X_j) \in
  \Gamma_\varepsilon^1)>1-\varepsilon,\qquad \inf_{j\in J} \mathbf
  P(\pi_2(\mathcal X_j)\in\Gamma_\varepsilon^2)>1-\varepsilon.
\end{align}
This is the same as (i) and (ii).

\subsection{Polynomials}
\label{sec:proofs5poly}
We prepare the proof of Theorem~\ref{T5} with some results on polynomials. We
show that polynomials separate points (Proposition~\ref{P:firstSecond1}) and
are convergence determining in $\mathbbm M^I$ (Proposition~\ref{l:aux}).

\begin{proposition}[Polynomials form an algebra that separates
  points]\mbox{}
  \label{P:firstSecond1}
  \begin{asparaenum}
  \item For $k=0,1,\dots,\infty$, the set of polynomials $\Pi^k$ is an
    algebra.  In particular, if $\Phi = \Phi^{n,\phi}\in\Pi^k_{n},
    \Psi=\Psi^{m,\psi}\in\Pi^k_m $, then
    \begin{align}
      \label{eq:301}
      (\Phi\cdot\Psi)(\smallu) = \langle \nu^\smallu, \phi \cdot
      (\psi\circ\rho_1^n)\rangle
    \end{align}
    with $\rho^n_1$ being the ``shift''
    \begin{align}
      \label{eq:rho0}
      \rho_1^n(\underline{\underline r}, \underline u) & =
      \big((r_{i+n,j+n})_{1\leq i<j}, (u_{i+n})_{i\geq 1}\big).
    \end{align}
  \item For all $k=1,2,\dots,\infty, \Pi^k$ separates points in $\mathbbm
    M^I$, i.e.\ for $\smallx, \smally\in\mathbbm M^I$ we have $\smallx =
    \smally$ iff $\Phi(\smallx) = \Phi(\smally)$ for all $\Phi\in\Pi^k$.
  \end{asparaenum}
\end{proposition}

\begin{proof}
  1. First, we note that the marked distance matrix distributions are
  exchangeable in the following sense: Let $\sigma: \mathbbm N\to\mathbbm N$
  be injective. Set
  \begin{align}
    \label{eq:29}
    R_\sigma: \begin{cases} \mathbbm R_+^{\binom{\mathbbm N}{2}} \times
      I^{\mathbbm N}& \to \mathbbm R_+^{\binom{\mathbbm N}{2}} \times
      I^{\mathbbm N}\\ \big((r_{ij})_{1 \leq i<j}, (u_k)_{k\geq 1}\big) &
      \mapsto \big((r_{\sigma(i)\wedge\sigma(j), \sigma(i)\vee\sigma(j)}),
      (u_{\sigma(k)})_{k\geq 1}\big).
    \end{cases}
  \end{align}
  Then, for $\smallx\in\mathbbm M^I$, we find that
  \begin{align}
    \label{eq:28a}
    (R_\sigma)_\ast \nu^\smallx = \nu^\smallx.
  \end{align}

  Next, we show that $\Pi^k$ is an algebra. Clearly, $\Pi^k$ is a linear space
  and $1\in\Pi^k$. Next consider multiplication of polynomials. By
  \eqref{eq:28a}, we find that $(\rho_1^n)_\ast \nu^\smallu = \nu^\smallu$. If
  $\Phi^{n,\phi}\in\Pi^k_{n}$, this implies
  \begin{equation}
    \label{eq:1}
    \begin{aligned}
      (\Phi\cdot\Psi)(\smallu) & = 
      \Bigl(\int \phi( \underline{\underline r}, \underline u)\nu^\smallu (d
      \underline{\underline r}, d\underline u) \Bigr)\cdot \Bigl(\int \psi(
      \rho_1^n(\underline{\underline r}, \underline u))\nu^\smallu (d
      \underline{\underline r},
      d\underline u) \Bigr)  \\
      & = \int \phi( \underline{\underline r}, \underline u)\psi(
      \rho_1^n(\underline{\underline r}, \underline u))\nu^\smallu (d
      \underline{\underline r}, d\underline u) = \langle \nu^\smallu, \phi
      \cdot (\psi \circ\rho_1^n) \rangle,
    \end{aligned}
  \end{equation}
  which shows that $\Pi^k$ is closed under multiplication as well.

  2. We turn to showing that $\Pi^k$ separates points. Recall that for
  $\smallx\in\mathbbm M^I$, the distance matrix distribution $\nu^\smallx$ is
  an element of $\mathcal M_1(\mathbbm R_+^{\binom{\mathbbm N}{2}} \times
  I^{\mathbbm N})$. On such product spaces, the set of functions
  \begin{align}
    \Big\{\phi(\underline{\underline r}, \underline u) = \prod_{k=1}^ng(u_k)
    \prod_{l=k+1}^n f_{kl}(r_{kl}): f_{kl}\in\overline{\mathcal C}^k(\mathbbm
    R_+), g_k\in \overline{\mathcal C}(I), n\in\mathbbm N\Big\}\subseteq \Pi^k
  \end{align}
  is separating in $\mathcal M_1(\mathbbm R_+^{\binom{\mathbbm N}{2}} \times
  I^{\mathbbm N})$ by Proposition 3.4.5 of \cite{EthierKurtz86}. If
  $\smallx\neq\smally$, we have $\nu^\smallx\neq\nu^\smally$ by
  Theorem~\ref{T1} and hence, there exists $\phi\in \Pi^k$ with $\langle \phi,
  \nu^\smallx\rangle \neq \langle \phi, \nu^\smally\rangle$ and hence $\Pi^k$
  separates points.
\end{proof}

\begin{proposition}[A convergence determining subset of $\Pi^\infty$]%
  \label{l:aux}\leavevmode\\
  There exists a countable algebra $\Pi^\infty_\ast \subseteq \Pi^\infty$ that
  is convergence determining in $\mathbbm M^I$, i.e.\ for $\smallx, \smallx_1,
  \smallx_2,\dots\in\mathbbm M^I$, we have
  $\smallx_n\xrightarrow{n\to\infty}\smallx$ iff
  $\Phi(\smallx_n)\xrightarrow{n\to\infty}\Phi(\smallx)$ for all
  $\Phi\in\Pi^\infty_\ast$.
\end{proposition}

\begin{proof}
  The necessity is clear. For the sufficiency argue as follows. Focus on the
  one-dimensional marginals of marked distance matrix distributions, which are
  elements of $\mathcal M_1(\mathbbm R_+^{\binom{\mathbbm N}{2}} \times
  I^{\mathbbm N})$ first. On the one hand by Lemma 3.2.1 of \cite{D93}, there
  exists a countable, linear set $V_{\mathbbm R_+}$ of continuous, bounded
  functions which is convergence determining in $\mathcal M_1(\mathbbm R_+)$,
  i.e.\ for $\mu, \mu_1, \mu_2,\dots\in\mathcal M_1(\mathbbm R_+)$ we have
  $\mu_n\xRightarrow{n\to\infty}\mu$ iff $\langle \mu_n, f\rangle
  \xrightarrow{n\to\infty} \langle \mu, f\rangle$ for all $f\in V_{\R_+}$. By
  an approximation argument, we can choose $V_{\mathbbm R_+}$ even such that
  it only consists of infinitely often continuously differentiable
  functions. On the other hand there exists a countable, linear set $V_I$ of
  continuous, bounded functions which is convergence determining in
  $I$. Without loss of generality, $V_{\mathbbm R_+}$ and $V_I$ are
  algebras. Since a marked distance matrix distribution $\nu^{\smallx}$ for
  $\smallx\in\mathbbm M^I$ is a probability measure on a countable product,
  Proposition 3.4.6 in \cite{EthierKurtz86} implies that the algebra
  \begin{align}\label{eq:721}
    V:=\Big\{\prod_{k=1}^n g_k(u_k) \prod_{l=k+1}^n f_{kl}(r_{kl}):
    n\in\mathbbm N, g_k\in V_I, f_{kl}\in V_{\mathbbm R_+}\Big\}
  \end{align}
  is convergence determining in $\mathcal M_1(\mathbbm R_+^{\binom{\mathbbm
      N}{2}}\times I^{\mathbbm N})$. In particular,
  \begin{align}\label{eq:722}
    \Pi_\ast^\infty := \{\smallx\mapsto \langle \nu^\smallx, \phi\rangle:
    \phi\in V\} \subseteq \Pi^\infty
  \end{align}
  is a countable algebra that is convergence determining. Indeed, for
  $\smallx, \smallx_1,\smallx_2,\dots \in\mathbbm M$, we have
  $\smallx_n\xrightarrow{n\to\infty} \smallx$ in the Gromov-weak topology iff
  $\nu^{\smallx_n}\xRightarrow{n\to\infty}\nu^\smallx$ in the weak topology on
  $\mathbbm R_+^{\binom{\mathbbm N}{2}}\times I^{\mathbbm N}$ iff $\langle
  \nu^{\smallx_n}, \phi\rangle \xrightarrow{n\to\infty} \langle \nu^\smallx,
  \phi\rangle$ for all $\phi\in V$.
\end{proof}

\subsection{Proof of Theorem~\ref{T5}}
\label{sec:proofT5}
By Theorem 3.4.5 of \cite{EthierKurtz86} and Proposition \ref{P:firstSecond1},
$\Pi^k$ is separating in $\mathcal M_1(\mathbb{M}^I)$.

We will show that $\Pi_\ast^\infty$ from Proposition~\ref{l:aux} is a
countable, convergence determining algebra in $\mathcal
M_1(\mathbb{M}^I)$. Recall $V$ and its ingredients, $V_I$ and $V_{\mathbbm
  R_+}$ from the proof of Proposition~\ref{l:aux}.
By Lemma~3.4.3 in \cite{EthierKurtz86}, we have that $\mathcal X_n
\xRightarrow{n\to\infty} \mathcal X$ iff (i) $\mathbf E[\Phi(\mathcal X_n)]
\xrightarrow{n\to\infty} \mathbf E[\Phi(\mathcal X)]$ for all
$\Phi\in\Pi_\ast^\infty$ and (ii) the family of distributions of $\{\mathcal
X_n: n\in\mathbbm N\}$ is tight. We will show that (i) implies (ii).

By Theorem~\ref{T4} we have to show that (i) implies that
\begin{align}\label{eq:997}
  \text{the family of distributions of }\{\pi_i(\mathcal X_n): n\in\mathbbm
  N\} \text{ is tight for }i=1,2.
\end{align}
Before we prove this relation we need some new objects and auxiliary facts.

For $(\underline{\underline r}, \underline u) \in \mathbbm
R_+^{\binom{\mathbbm N}{2}} \times I^{\mathbbm N}$ and $\varepsilon>0$, we set
\begin{equation}
  \label{eq:999}
  \begin{aligned}
    v(\underline{\underline r}, \underline u) &:= u_1,\\
    w(\underline{\underline r}, \underline u) &:= r_{12},\\
    z_\varepsilon(\underline{\underline r}, \underline u) & :=
    \limsup_{n\to\infty} \tfrac{1}{n} \sum_{i=2}^n 1_{\{r_{1n}<\varepsilon\}}.
  \end{aligned}
\end{equation}
Moreover, for a random variable $\mathcal Y$ with values in $\mathbbm M^I$, we
define $(\underline{\underline R}, \underline U)^{\mathcal Y}$ as the random
variable with values in $\mathbbm R_+^{\binom{\mathbbm N}{2}} \times
I^{\mathbbm N}$, such that given $\mathcal Y = \smally$,
$(\underline{\underline R},\underline U)^{\mathcal Y}$ has distribution
$\nu^{\smally}$. We have
\begin{equation}
  \label{eq:trans1}
  \begin{aligned}
    \mathbf E[\phi((\underline{\underline R}, \underline U)^{\mathcal X_n})] &
    = \mathbf E\big[\mathbf E[\phi((\underline{\underline R}, \underline
    U)^{\mathcal X_n})|\mathcal X_n]\big] \\ & = \mathbf E[\langle
    \nu^{\mathcal X_n}, \phi\rangle] \xrightarrow{n\to\infty} \mathbf
    E[\langle \nu^{\mathcal X}, \phi\rangle] = \mathbf
    E[\phi((\underline{\underline R}, \underline U)^{\mathcal X})],
  \end{aligned}
\end{equation}
for all $\phi \in V$ by Assumption (i). Since $V$ is convergence determining
in $\mathcal M_1(\mathbbm R^{\binom{\mathbbm N}{2}} \times I^{\mathbbm N})$,
we note that
\begin{align}\label{eq:trans0}
  (\underline{\underline R}, \underline U)^{\mathcal X_n}
  \xRightarrow{n\to\infty} (\underline{\underline R}, \underline U)^{\mathcal
    X}.
\end{align}

In order to show \eqref{eq:997} for $i=1$, by Theorem 3 of
\cite{GPWmetric09}, we need to show that \eqref{eq:trans0} implies
\begin{enumerate}
\item[(a)] $\{w({\underline{\underline R}}^{\mathcal X_n}): n\in\mathbbm N\}$
  is tight
\item[(b)] For all $\varepsilon>0$ there exists $\delta>0$ such that
  $\limsup_{n\to\infty} \mathbf P(z_\varepsilon(\underline{\underline
    R}^{\mathcal X_n})<\delta) < \varepsilon$.
\end{enumerate}
For (a), note that by \eqref{eq:trans1}
\begin{equation}
  \label{eq:996}
  \begin{aligned}
    \mathbf E[f(w(({\underline{\underline R}}, \underline U)^{\mathcal X_n})]
    \xrightarrow{n\to\infty} \mathbf E[f(w(({\underline{\underline R}},
    \underline U)^{\mathcal X})]
  \end{aligned}
\end{equation}
\sloppy for all $f\in V_{\mathbbm R_+}$. Hence, since $V_{\mathbbm R_+}$ is
convergence determining in $\mathbbm R_+$, $w(({\underline{\underline R}},
\underline U)^{\mathcal X_n}) \xRightarrow{n\to\infty}
w(({\underline{\underline R}}, \underline U)^{\mathcal X})$, and in
particular, (a) holds.

\medskip

\noindent
For (b), consider the distribution of $z_\varepsilon((\underline{\underline
  R}, \underline U)^{\mathcal X})$. Since the single random variable $\mathcal
X$ is tight in $\mathbbm M$, by Theorem 3 of \cite{GPWmetric09}, we find
$\delta>0$ such that $\mathbf P(z_\varepsilon((\underline{\underline R},
\underline U)^{\mathcal X})<\delta) < \varepsilon$ and
$z_\varepsilon((\underline{\underline R}, \underline U)^{\mathcal X})$ does
not have an atom at $\delta$. By the latter property, the set
$A:=\{(\underline{\underline r}, \underline u):
z_{\varepsilon}(\underline{\underline r}, \underline u) < \delta\}$ has the
property $\partial A \subseteq \{(\underline{\underline r}, \underline u):
z_{\varepsilon}(\underline{\underline r}, \underline u) = \delta\}$ and thus,
$\mathbf P(({\underline {\underline R}, \underline U)}^{\mathcal X}
\in \partial A)=0$, since $z_\varepsilon(\underline{\underline R}, \underline
U)^{\mathcal X})$ does not have an atom at $\delta$. By the Portmanteau
Theorem,
\begin{equation}
  \label{eq:995}
  \begin{aligned}
    \mathbf P(z_\varepsilon((\underline{\underline R}, \underline U)^{\mathcal
      X_n})<\delta) & = \mathbf P((\underline{\underline R}, \underline
    U)^{\mathcal X_n} \in A) \xrightarrow{n\to\infty} \mathbf
    P((\underline{\underline R}, \underline U)^{\mathcal X} \in A) \\ & =
    \mathbf P(z_\varepsilon((\underline{\underline R}, \underline U)^{\mathcal
      X})<\delta)<\varepsilon.
  \end{aligned}
\end{equation}
This shows (b).

In order to obtain \eqref{eq:997} for $i=2$, note that $v_\ast
\nu^{\mathcal X_n} \in \mathcal M_1(I)$ is the first moment measure of
the distribution of the $\mathcal M_1(I)$-valued random variable
$\pi_2(\mathcal X_n)$ and recall that tightness in $\mathcal
M_1(\mathcal M_1(I))$ is implied by tightness of the first moment
measure. By \eqref{eq:trans1}, we find that for $g\in V_I$
\begin{align}
  \mathbf E[g(v((\underline{\underline R}, \underline U)^{\mathcal X_n}))]
  \xrightarrow{n\to\infty} \mathbf E[g(v((\underline{\underline R}, \underline
  U)^{\mathcal X}))],
\end{align}
so $v((\underline{\underline R}, \underline U)^{\mathcal X_n})
\xRightarrow{n\to\infty}v((\underline{\underline R}, \underline
U)^{\mathcal X})$ and, in particular, \eqref{eq:997} holds for $i=2$.

\subsubsection*{Acknowledgments}
AD and PP acknowledge support from the Federal Ministry of Education
and Research, Germany (BMBF) through FRISYS (Kennzeichen 0313921) and
AG from the DFG Grant GR 876/14.  Part of this work has been carried
out when AD was taking part in the Junior Trimester Program
Stochastics at the Hausdorff Center in Bonn: hospitality and financial
support are gratefully acknowledged.

\bibliographystyle{plain}


\end{document}